\documentclass[11pt]{amsart}

\usepackage{amscd,amsmath,amssymb,fancyhdr,color}

\newcommand{\Ker}{\text{Ker}}

\newcommand{\RR}{\mathbb{R}}

\numberwithin{equation}{section}

\def\eqref#1{(\ref{#1})}

\newcommand{\g}{{\mathfrak g}}

\newcommand{\Z}{{\mathbb Z}}
\newcommand{\C}{{\mathbb C}}
\newcommand{\R}{{\mathbb R}}

\renewcommand{\H}{{\mathbb H}}

\def\1{\sqrt{-1}\:}

\newcommand{\cntrct}                
{\hspace{2pt}\raisebox{1pt}{\text{$\lrcorner$}}\hspace{2pt}}
\newcommand{\arrow}{{\:\longrightarrow\:}}

\renewcommand{\bar}{\overline}
\renewcommand{\phi}{\varphi}
\renewcommand{\epsilon}{\varepsilon}

\renewcommand{\leq}{\leqslant}

\newcommand{\id}{\operatorname{\text{\sf id}}}

\newcommand{\Diff}{\operatorname{Diff}}

\newcommand{\Deck}{\operatorname{Deck}}
\newcommand{\Ver}{\operatorname{Vert}}
\newcommand{\Ham}{\operatorname{Ham}}

\renewcommand{\Im}{\operatorname{Im}}

\newcounter{Mycounter}[section]
\newcounter{lemma}[section]
\newcounter{claim}[section]

\newcounter{sublemma}[section]

\newcounter{corollary}[section]

\newcounter{theorem}[section]

\newcounter{conjecture}[section]

\newcounter{proposition}[section]

\newcounter{definition}[section]

\renewcommand{\thedefinition} {{Definition~\thesection.\arabic{definition}}}

\newcommand{\definition}{%

     \setcounter{definition}{\value{Mycounter}}

     \refstepcounter{definition}

     \stepcounter{Mycounter}

     {\medskip\noindent \bf \thedefinition:\ }}

\newcounter{example}[section]

\newcounter{remark}[section]

\renewcommand{\theremark}{{Remark \thesection.\arabic{remark}}}

\newcommand{\remark}{%

     \setcounter{remark}{\value{Mycounter}}

     \refstepcounter{remark}

     \stepcounter{Mycounter}

     {\medskip\noindent \bf \theremark:\ }}

\newcounter{problem}[section]

\newcounter{question}[section]

\makeatletter

\@addtoreset{equation}{section}

\@addtoreset{footnote}{section}

\makeatother

\def\blacksquare{\hbox{\vrule width 5pt height 5pt depth 0pt}}

\def\endproof{\blacksquare}

\begin{document}

\newpage

\title{Locally conformally symplectic bundles}
\author{Alexandra Otiman}

\date{\today}
\address{Institute of Mathematics ``Simion Stoilow'' of the Romanian Academy\\
21, Calea Grivitei Street, 010702, Bucharest, Romania {\em and} University of Bucharest, Faculty of Mathematics and Computer Science, 14 Academiei Str., Bucharest, Romania}
\email{alexandra\_otiman@yahoo.com}

\abstract A locally conformally symplectic (LCS) form is an almost symplectic form $\omega$ such that  a closed one-form $\theta$ exists with $d\omega=\theta\wedge\omega$. A fiber bundle with LCS fiber $(F, \omega,\theta)$ is called LCS if the transition maps are diffeomorphisms of $F$ preserving $\omega$ (and hence $\theta$). In this paper, we find conditions for the total space of an LCS fiber bundle to admit an LCS form which restricts to the LCS form of the fibers. This is done by using the coupling form introduced by Sternberg and Weinstein,  \cite{gls}, in the symplectic case. The construction is related to an adapted Hamiltonian action called twisted Hamiltonian which we study in detail. Moreover, we give examples of such actions and discuss  compatibility properties with respect to LCS reduction of LCS fiber bundles. We end with a glimpse towards the locally conformally K\"ahler case.\\[.1in]

\endabstract

\maketitle

\tableofcontents

\section{Introduction}

Locally conformally symplectic manifolds (shortly LCS), a subclass of almost symplectic manifolds,  were introduced by P. Liebermann in \cite{lib}, but their modern study began with J. Lefebvre, \cite{lef}, and especially with I. Vaisman in \cite{v}.
From a conformal viewpoint,  LCS manifolds might be viewed as  closest to symplectic manifolds, meaning they are endowed with a nondegenerate 2-form $\omega$ and a covering $(U_{\alpha})$ of the manifold such that there exist smooth functions $f_{\alpha}$ on $U_{\alpha}$ with $d(e^{-f_\alpha}\omega)=0$ (see also \cite{lee}, where such an equation was first considered). As $\omega$ is non-degenerate, the local 1-forms $df_{\alpha}$ glue up to a global closed form and motivate the following two equivalent definitions:
 
\begin{definition} A manifold $M$ is called {\em locally conformally symplectic} if there exists a non-degenerate 2-form $\omega$ and a closed 1-form $\theta$ such that $d \omega = \theta \wedge \omega$.
\end{definition}

\begin{definition} A manifold $M$ is locally conformally symplectic if there exists a symplectic covering ($\tilde{M}, \Omega$) of $M$ with the deck group acting by positive homotheties of $\Omega$. 
\end{definition}

\begin{remark}
The closed one-form $\theta$ is called the {\em Lee form} (after the name of  H.C. Lee, \cite{lee}). Defining  $d_\theta=d - \theta \wedge \cdot$, the LCS condition reads $d_{\theta}\omega=0$, formally identical with the symplectic one. Note that $d\theta=0$ implies $d_\theta^2=0$ and in fact the operator $d_\theta$ produces the Morse-Novikov cohomology, also called Lichnerowicz or twisted cohomology by several authors.
\end{remark}

\smallskip

The second  definition tells us that it is natural to consider the following object associated to a symplectic covering (on which the Lee form is exact) of an LCS manifold. Let $\chi : \Deck(\tilde{M}/M) \rightarrow R^{+}$ be the map associating to an element $\gamma$ of $\Deck(\tilde{M}/M)$ the scalar factor of the corresponding homothety. One can easily show that $\chi$ is a group morphism. We shall call an object $\alpha$ on $\tilde{M}$ with the property that $\gamma^*\alpha=\chi(\gamma)\alpha$ \noindent{\em{automorphic}}. For instance, the symplectic form on $\tilde{M}$ is automorphic. Let us denote by $L$ the line bundle $\tilde{M}\times_{\chi}\mathbb{R}$. It is clear that $L$ is isomorphic to the trivial bundle, but any automorphic object on $\tilde{M}$ can be regarded as an object on $M$ with values in $L$.

Many examples of LCS manifolds come from locally conformally K\" ahler (LCK) geometry by forgetting the complex structure (see \cite{do}, \cite{ov} and the examples and references therein),  but  examples which cannot be  LCK where constructed too, {\em e.g.} \cite{bm}. 

\smallskip

The aim of this paper is to extend the following notion: 

\begin{definition}  A fiber bundle with symplectic fiber $(F, \omega)$, $F \rightarrow M \rightarrow B$, is called a {\em symplectic bundle} if the transition maps are diffeomorphisms of $F$ preserving $\omega$. 
\end{definition}

\smallskip

The primary motivation for introducing symplectic bundles was providing new examples of symplectic manifolds. Weinstein, Sternberg, Guillemin and Lerman studied the possibility of endowing the total space of such a fibration with a symplectic form which coincides with the symplectic form of the fibres, when restricted to any fibre. In \cite{s}, Sternberg constructed  a closed two-form with such properties, provided some additional requirements are considered, which he called {\em coupling form}. Weinstein further related it to the notion of {\em fat connection}  in order to obtain a symplectic form on the fiber bundle. The coupling form has been extended to more general settings. In \cite{vo}, Vorobjev introduced and studied a subclass of Poisson structures on fiber bundles, called coupling tensors. Later, Vaisman extended Vorobjev's results from fiber bundles to foliated manifolds and moreover, for Jacobi structures (see \cite{va2}). Fat connections were studied in \cite{w} and the definition is the following:

\definition Let $G$ be a Lie group with Lie algebra $\g$ and let $P$ be a $G$-principal bundle. Let $A$ be a $\g$ valued two-form which is the curvature of a connection form $a$ in $\Omega^1(P, \g)$. A point $f$ in $g^*$ is called {\emph{fat}} if $f \circ A: \Ker a \times \Ker a \rightarrow \mathbb{R}$ is nondegenerate.

\smallskip

We are going to extend the following result to the LCS case.
\begin{theorem} {\rm \cite{s,w}}  Let $(F, \omega)$ be a symplectic manifold with a Hamiltonian action of a Lie group $G$ on $F$. If $\mu: F \rightarrow g^*$ is the momentum map, then any  connection on a $G$-principal bundle $P$ which is fat at all the points in $\mu(F)$  induces a symplectic form on $P \times_{G}F$. 
\end{theorem}

\smallskip

In order to find the proper analogue in LCS setting, we define locally conformally symplectic bundles and we use, instead of Hamiltonian actions, {\em twisted Hamiltonian actions} defined by Vaisman in \cite{v}. 

As the LCS condition is conformally invariant (if $\omega$ is LCS, then any $e^f\omega$ is LCS too), one could work with conformal actions w.r.t the LCS form, as Haller and other authors do.  However, we shall work in a more restrictive frame by requiring  the action of a Lie group to preserve the LCS form. Since we shall usually work with compact groups and by an averaging procedure one can find an invariant LCS form, we shall consider only actions preserving the LCS form. Moreover, this is very helpful in extending the result of Sternberg and Weinstein, as we need a field of $G$-invariant forms on the fibers. 

Let $G$ be a Lie group with Lie algebra $\g$.

\begin{definition}  Let $\rho:G\arrow\Diff(M)$ be an  action of $G$ on an LCS manifold $(F, \omega, \theta)$. If $\omega$ is $G$--invariant and if $i_{\rho(v)}\omega$ is $d_{\theta}$ - exact,  the action is called {\em twisted Hamiltonian}. A choice of a smooth function $\psi: \g \rightarrow \mathcal{C}^{\infty}(F)$ such that $i_{\rho(v)}\omega=d_{\theta}\psi(v)$ will give a momentum map $\mu: F \rightarrow g^{*}$.
\end{definition}

\smallskip

The main result we prove in this paper is:

\begin{theorem} \label{main}Let $(F, \omega, \theta)$ be a locally conformally symplectic manifold not globally conformally symplectic and let  $G$ be a Lie group acting on $F$ by diffeomorphisms preserving $\omega$ (and hence $\theta$). If the action of $G$ is twisted Hamiltonian and if $\mu:F \rightarrow g^{*}$ is a momentum map of the action, then any connection on a $G$ - principal bundle $P$ which is fat at $\mathrm{Im} \mu$ induces an LCS structure on $P \times_{G} F$. 
\end{theorem}

\smallskip

We shall give the proof in Sections 3 and 4 and devote Section 5 to a discussion of twisted Hamiltonian actions on LCS manifolds, which we study in detail, one of the important results being \ref{minimal}. Section 5 also contains several worked examples. In Section 6, we study compatibility properties for the fiber and the total space of an LCS fibration with respect to LCS reduction, confining ourselves to actions of abelian groups. Finally, in Section 7 we investigate the possibility of replicating the construction of a coupling form in the LCK setting, but the complex case seems to be very rigid and the general result is not satisfactory.

\section{Preliminaries}

The definition of symplectic bundle was motivated by the following proposition noted in \cite{ms}.

\begin{proposition}\label{compsym} Let $F \rightarrow M \rightarrow B$ be a fiber bundle with compact fiber and let $M$ be endowed with a symplectic form which restricted to the fibers remains symplectic. Then $F \rightarrow M \rightarrow B$ admits the structure of a symplectic fiber bundle. 
\end{proposition}

\smallskip

In the LCS setting, we have the following analogue which leads to a possible definition of LCS bundles: 

\begin{proposition}\label{complcs} Let $F \rightarrow M \xrightarrow{\pi} B$ be a fiber bundle with compact fiber such that $M$ is LCS with the Lee form $\Theta$ and the LCS form $\Omega = d_{\Theta}\eta$. If the restriction of $d_\Theta\eta$  to any fiber is an LCS form, then there exists an LCS structure $(\omega, \theta)$ on $F$ and there exist trivializations for the fiber bundle such that the transition maps preserve the conformal class of $\omega$.  
\end{proposition}

\smallskip

The proof of \ref{compsym} rests on using the Moser trick for symplectic forms. Similarly, the proof of \ref{complcs} uses te following  LCS version of Moser trick:

\begin{proposition}\label{moserlcs}{\rm (\cite[Corollary 5.]{bk})} 
Let $\omega_t$ be a smooth family of LCS forms on a compact manifold such that the corresponding Lee forms $\theta_t$ are cohomologous. If there exists a smooth family $\eta_t$ with $\omega_t=d_{\theta_t}\eta_t$, then there exists an isotopy $\phi_t$ such that $\phi_t^*\omega_t$ is conformally equivalent to $\omega_0$ for all $t$.
\end{proposition}

\medskip

\noindent{\em Proof of \ref{complcs}.} Let $\phi_{\alpha}: \pi^{-1}(U_{\alpha}) \rightarrow U_\alpha \times F$ be trivializations for a covering $(U_\alpha)$ of $B$. We want to show there exist diffeomorphisms $\tilde{\phi}_{\alpha}: U_{\alpha} \times F \rightarrow U_{\alpha} \times F$ so that the transition maps of $\tilde{\phi}_{\alpha} \circ \phi_{\alpha}$ preserve the conformal class of some standard LCS form on $F$. 

To begin with, we denote henceforth the fiber over a point $b$ by $F_b$ and by $\theta_{b}$ and $\eta_{b}$ we shall understand the restriction of $\Theta$, respectively $\eta$ to the fiber $F_{b}$. For any $b$ in $U_{\alpha}$, one can show that the 1-forms $\tilde{\theta}^{\alpha}_b:=(\phi_{\alpha})_{*}(\theta_b)$ define a family of closed cohomologuous forms depending smoothly on the parameter $b$ in $U_{\alpha}$. The 1-forms $\tilde{\eta}^{\alpha}_b=(\phi_{\alpha})_{*}(\eta_b)$ also define a smooth family. So, $F$ is endowed with a smooth family of LCS forms $\omega^{\alpha}_{b}:=d_{\tilde{\theta}^{\alpha}_b}\tilde{\eta}^{\alpha}_b$. 

We now apply \ref{moserlcs}. 
Without loss of generality, we can suppose we  have chosen from the beginning a contractible covering $(U_{\alpha})$ such that the intersection of all the open sets from the covering is non-empty. Let $x_0$ be a point in this intersection. 

We iterate the application of the Moser trick in order to apply it to contractible sets. Thus, we find an isotopy $k^{\alpha}_{b}: F \rightarrow F$ with $(k^{\alpha}_{b})^*\omega^{\alpha}_{x_0}=f\omega^{\alpha}_b$, where $f$ is a smooth function on $F$. Then $(k^{\alpha}_{b})^*[\omega^{\alpha}_{x_0}]_c=[\omega^{\alpha}_b]_c$ (we denote by $[ \cdot ]_c$ the conformal class). We define the diffeomorphism $\psi_{\alpha}: U_{\alpha} \times F \rightarrow U_{\alpha} \times F$ given by $\psi_{\alpha}(b, f)=(b, k^{\alpha}_b(f))$. 

The new trivialization maps $(\tilde{\psi}_{\alpha} = \psi_{\alpha} \circ \phi_{\alpha})_{\alpha}$ have the property that $(\tilde{\psi}_{\alpha} \circ \tilde{\psi}^{-1}_{\beta|\{b\} \times F})^{*}[\omega^{\alpha}_{x_0}]_c=[\omega^{\beta}_{x_0}]_c$. Let us fix some open cover set $U_{\alpha_0}$ and define $\chi_{\alpha}: U_\alpha \times F \rightarrow U_\alpha \times F$, $\chi_\alpha (x, f)=(x, (\phi_{\alpha_0} \circ \phi^{-1}_{\alpha})_{|\{x_0\} \times F}(f))$. Then, it is easy to check that the trivialization $(\Phi_{\alpha})_{\alpha}$ defined by $\Phi_\alpha:=\tilde{\psi}_\alpha \circ \chi_{\alpha}$ satisfy $(\Phi_\alpha \circ \Phi^{-1}_\beta)^*_{|\{b\} \times F}[\omega^{\alpha_0}_{x_0}]_c=[\omega^{\alpha_0}_{x_0}]_{c}.$
Now it is clear that the LCS form we have to consider on $F$ is $\omega=\omega^{\alpha_0}_{x_0}$ and thus we constructed a trivialization, namely $\Phi_{\alpha}$, whose transition maps preserve the conformal class of $\omega$ (and implicitly de Rham class of $\theta$). 
\endproof

\smallskip

The definition that now appears to be most natural for LCS fiber bundles is the one requiring for a fiber bundle with LCS fiber the existence of a trivialization with transition  maps preserving the conformal class of the LCS form. Nevertheless, for the purpose of extending Sternberg's coupling form to the LCS setting, it will be more convenient to work with a restrictive definition of LCS bundles.

\smallskip  

Let $(F, \omega, \theta)$ be a LCS manifold. We denote by $\Diff(F, \omega, \theta)$  the group of diffeomorphisms of $F$ preserving $\omega$. Note that, as $\omega$ is non-degenerate and $d\omega=\theta\wedge\omega$, if $f\in\Diff(M)$ preserves $\omega$, it preserves $\theta$ too. 

\begin{definition} A fiber bundle $F \rightarrow M \rightarrow B$ with $(F, \omega, \theta)$ LCS is  locally conformally symplectic if there exists trivializations with transition maps in $\Diff(F, \omega, \theta)$.
\end{definition}

\begin{remark} If $F \rightarrow M \rightarrow B$ is an LCS fiber bundle, the fiber at each point comes naturally with an LCS form $\omega_{b}$ with the Lee form $\theta_{b}$. Indeed, if $\phi_{U}: \pi^{-1}(U) \rightarrow U \times F$ is a local trivialization, $(\phi_{U}|_{F_{b}})^{*} \omega$ is a well defined LCS form on $F_{b}$, whose Lee form is $(\phi_{U}|_{F_{b}})^{*}\theta$. We can define $\omega_{b}$ and $\theta_{b}$ with any other trivializing map $\phi_{V}$ with $V$ containing $b$, since the transition maps take value in $\Diff(F, \omega, \theta)$. 
\end{remark}

\begin{remark} If $(F, \omega, \theta)$ is an LCS manifold, $G$ a group acting on $F$ by diffeomorphisms from $\Diff(F, \omega, \theta)$ and $P$ a $G$ - principal bundle, $P \times_{G} F$ is an LCS bundle of fibre $F$ and base $P/G$. These are the fiber bundles which we shall mostly consider. If $G$ is a compact group, then an averaging procedure as described in \cite{ov2} will yield on $F$ a $G$-invariant LCS form, regardless of the nature of the action of $G$. Thus, restricting ourselves to bundles with LCS form - preserving structural group does not represent a severe constraint for our purpose.
\end{remark}

\begin{example}   The Hopf fibration:
\begin{equation}\label{s7}
S^1 \times S^3 \rightarrow S^1 \times S^7 \rightarrow S^4.
\end{equation}
is an example of a LCS fiber bundle.
\end{example}

Since \eqref{s7} is a prototype of the objects we are interested in, we shall provide more details in order to understand this example. For this, we have to see better what happens with the Hopf fibration $S^3 \rightarrow S^7 \rightarrow S^4$. Let us look at $S^7$ as the set $\{(a, b) \in \mathbb{H}^2 \mid |a|^2+|b|^2=1\}$, where $\mathbb{H}$ denotes the quaternionic numbers. We consider the natural map from $S^7$ to the quaternionic projective space $\mathbb{HP}^1$, taking $(a, b)$ to its equivalence class $[a, b]$. The fiber over $[a, b]$ consists of all the points on the sphere $(\lambda a, \lambda b)$, with $\lambda \in \mathbb{H}^*$ and $|\lambda|=1$. This identifies the fiber with $S^3$ and it is in fact, the Hopf fibration, via the isomorphism of $S^4$ with $\mathbb{HP}^1$. 
Take now the sets $U_1=\{[x_1, x_2] \in \mathbb{HP}^1 \mid x_1 \neq 0\}$ and $U_2=\{[x_1, x_2] \in \mathbb{HP}^1 \mid x_2 \neq 0\}$. They represent the trivializations sets of this fiber bundle and the transition maps are given by:
$$U_1 \cap U_2 \times S^3 \rightarrow U_1 \cap U_2 \times S^3$$
$$([1, \lambda], s) \mapsto ([1, \lambda], s \cdot \tfrac{\lambda}{|\lambda|})$$
This is just the multiplication with an element from $S^3$ and it preserves the standard contact form of $S^3$, which we shall denote by $\alpha$. We consider now the trivial product with $S^1$ to obtain \eqref{s7}. Obviously, the transition maps let the volume form $\nu$ of the circle intact. Consequently, they preserve the form $d_{\nu}\alpha$, which represents the LCS form of $S^1 \times S^3$.

Moreover, the LCS form on $S^1 \times S^7$, given by $d_{\nu}\tilde{\alpha}$, where $\tilde{\alpha}$ is the standard contact form of $S^7$, restricts on the fibers to the LCS form of $S^1 \times S^3$, since $\tilde{\alpha}_{|S^3} = \alpha$. So this fiber bundle has the property  we desire and try to depict in the paper, namely an LCS bundle with an LCS form on the total space which restricts {\em{well}} on the fibers.

\definition Let $F \rightarrow M \rightarrow B$ be an LCS bundle. We say that $(\Omega, \Theta)$ on $M$ is an {\em extension} of $(\omega, \theta)$ if  $i_b^*\Omega=\omega_b$ and $i_{b}^*\Theta=\theta_b$ for any $b$ in $B$. 

\smallskip

A simple observation is that if $\Omega$ is an LCS form on the total space of a fibre bundle that restricts to the LCS forms of the fibres, then also its Lee form $\Theta$ will restrict to the Lee forms of the fibres, since the LCS form uniquely determines  the Lee form.

\section{Twisted closed extensions}

In the symplectic setting, the first step towards finding a symplectic form on the total space was looking for a closed two-form which restricts well on the fibers. The following result is well known (we cite it according to \cite{glsw}):

\begin{theorem} {\rm (Thurston)}\label{thurston} Let $F \rightarrow M \rightarrow B$ be a differentiable fibre bundle carrying a field $\omega$ of $p$-forms on the vertical bundle, defining a closed form on each fibre. Then there exists a closed $p$-form $\Omega$ on $M$ which restricted on the fibers is equal to $\omega$ if and only if there exists a de Rham cohomology class $c$ on $M$ such that its restriction to each fibre coincides with the class determined by $\omega$.
\end{theorem}

\smallskip

We now find a similar criterion for LCS forms. 

As in the symplectic case, \cite{glsw}, we could directly apply Proposition 2.13 in \cite{hattori}, which is very general,  but has a very involved proof. For the sake of completeness, we shall give a direct and short proof using Hodge theory instruments, which we believe is interesting in itself.

\begin{theorem}\label{ext} Let $F \rightarrow M \rightarrow B$ be an LCS fiber bundle with $F$ compact and let $\Theta$ be a one-form such that $\Theta_{|F_b}=\theta_b$. Then there exists a $d_{\Theta}$-closed form $\Omega$ on $M$ with $(\Omega, \Theta)$ an extension of $(\omega, \theta)$ if and only if there exists a class $c$ in the twisted cohomology $H^2_{\Theta}(M)$ such that for any $b$, $i_{b}^*c=[\omega_b]$ in $H^{2}_{\theta_{b}}(F_b)$, where $i_{b}: F_b \rightarrow M$ is the natural inclusion of the fibre.
\end{theorem}

\smallskip

{The next lemma is crucial for the proof of \ref{ext}:}

\begin{lemma} Let $F$ be a compact manifold and $\theta$ a closed 1-form. If $\{\omega_{t}\}_{t}$ is a smooth family of $d_{\theta}$-exact forms indexed by an open set in $\mathbb{R}^n$, then there exists a smooth family of forms $\{\psi_t\}_{t\in\RR}$ such that $\omega_{t}=d_{\theta}\psi_t$ for any $t$.
\end{lemma}

\begin{proof} 
We choose a Riemannian metric g  and denote by $d_{\theta}^*$ the adjoint of $d_{\theta}$ with respect to $g$. The Laplacian associated to $d_{\theta}$ is $\Delta_{\theta}=d_{\theta}d_{\theta}^*+ d_{\theta}^{*}d_{\theta}$ and it is an elliptic operator. The following Hodge decomposition for the twisted differential holds: 
\begin{equation}\label{hodgez}
\Omega^{p}(F)=\mathcal{H}^{p}_{\theta}(F) \oplus d_{\theta}\Omega^{p-1}\oplus d_{\theta}^{*}\Omega^{p+1}
\end{equation}
where $\mathcal{H}^{p}_{\theta}=\mathrm{Ker}\Delta_{\theta}$. We observe that each $d_{\theta}$-exact form belongs to $d_{\theta}d_{\theta}^*\Omega^{p}$ by plugging in \eqref{hodgez} the Hodge decomposition for $\Omega^{p-1}$. So each $d_{\theta}$-exact form has a  primitive which is $d_{\theta}^*$-exact, hence the following map is surjective:
$$d_{\theta}: d_{\theta}^{*}\Omega^{p} \rightarrow d_{\theta}\Omega^{p-1}.$$
 Moreover, we prove that this primitive is unique, resulting in the bijectivity of the map above. Indeed, let us assume that $d_{\theta}d_{\theta}^*\eta_1=d_{\theta}d_{\theta}^*\eta_2$, then 
$$\langle d_{\theta}^*\eta_1, d_{\theta}^*\psi\rangle=\langle d_{\theta}^*\eta_2, d_{\theta}^*\psi\rangle.$$
But because of the Hodge decomposition this implies that $$\langle d_{\theta}^*\eta_1, \eta\rangle=\langle d_{\theta}^*\eta_2, \eta\rangle$$ for any $\eta$, hence $d_{\theta}^*\eta_1=d_{\theta}^*\eta_2$. 

We denote by $G_{\theta}:  (\mathcal{H}^{p}_{\theta}(F))^{\perp}\rightarrow (\mathcal{H}^{p}_{\theta}(F))^{\perp}$ the Green operator associating to $\alpha$ the unique solution $\phi$ of $\Delta_{\theta} \phi = \alpha$. By using the commutation of $G_{\theta}$ with $d_{\theta}$ and $d_{\theta}^*$ (see Proposition 6. 10 in \cite{fw}), one can easily show that if $\omega=d_{\theta}\eta$, then $\omega=d_{\theta}d_{\theta}^{*}G_{\theta} \omega$ (note that $\omega$ belongs to $ (\mathcal{H}^{p}_{\theta}(F))^{\perp}$). Therefore, the primitive of a $d_{\theta}$-exact form $\omega$ can be chosen to be $d_{\theta}^{*}G_{\theta}\omega$. Since $d_{\theta}^*$ and $G_{\theta}$ are smooth operators, the family of primitives $\psi_t=d_{\theta}^{*}G_{\theta}\omega_t$ is smooth.
\end{proof}

\smallskip

{We now proceed  with  the:

\medskip

\noindent{\em Proof of  \ref{ext}.}
The direct implication is obvious. For the converse, consider $c$ in $H^2_{\Theta}(M)$ with $i_{b}^*c=[\omega_b]$ in $H^2_{\theta_b}(F)$ and fix $\Omega_0$ in $c$. Then there exists a one-form $\psi_b$ on $F_b$ such that $i_b^*\Omega_0 - \omega_b=d_{\theta_b}\psi_b$. We now restrict only to those $b$ belonging to a trivialization open set $U$ and consider $\Phi_U : \pi^{-1}(U) \rightarrow U \times F$ a trivialization map. Since the map $\Phi_U$ defines an isomorphism between $F_b$ and $F$ for any $b$ in $U$, $i_b^*c - \omega_b$ is a smooth family of vertical forms and $\Phi_{U*}\theta_b=\theta$,  $\Phi_{U*}d_{\theta_b}\psi_b$ defines a smooth family of $d_{\theta}$-exact forms on F. 
By using the previous lemma, we get that the collection $\psi_b$ can be chosen smooth, so there exists $\psi_U$  a one-form on $\pi^{-1}(U)$ such that $i_b^*\psi_{U}=\psi_b$. We define the following two-form:
$$\Omega = \Omega_0 - d_{\Theta} \sum_{U} (\pi^* \rho_U) \psi_U$$ where $\{\rho_U\}_U$ is a partition of unity on $B$.
By restricting $\Omega$ to $F_b$ we get that:
\begin{align*}
i_{b}^*\Omega &=i_b^{*}\Omega_0-d_{\theta_b}\sum_U(\rho_U(b))\psi_b\\
                              &=i_b^{*}\Omega_0-d_{\theta_b}\psi_b\\
                              &=\omega_b.
\end{align*}
This completes the proof.
\endproof

\begin{corollary} Let $(F, \omega, \theta)$ be a compact locally conformally symplectic manifold and let $\omega=d_{\theta}\eta$. If $F \rightarrow M \rightarrow B$ is an LCS bundle and there is a cohomology class $c$ in $H^{1}_{dR}(M)$ such that for any $b$ in $B$, $i_{b}^*c=[\theta_b]$, one can find a twisted closed extension for $(\omega, \theta)$. 
\end{corollary}

\begin{proof} The proof is obvious since once found a 1-form $\Theta$ on $M$ extending $\theta$ on the fibers, the condition in  \ref{ext} applied to $\Theta$ is automatically satisfied, because the class of $\omega_{b}$ in $H^{2}_{\theta_b}(F_{b})$ is 0 and is trivially extended by the class 0 in $H^{2}_{\Theta}(M)$. But such a $\Theta$ exists by applying  \ref{thurston}.
\end{proof}

\begin{proposition}\label{clara} Let $(F, \omega, \theta)$ be a compact LCS manifold with $\omega=d_{\theta}\eta$, $G$ a Lie group acting on $F$ by preserving $\omega$ and $P$ a principal $G$-bundle. If $\tilde{\rho}: g \rightarrow \mathcal{X}(F)$ is the infinitesimal action induced by the action of $G$ on $F$ on the Lie algebra $g$ and $\theta(\tilde{\rho}(v))=0$ for any $v$ in $g$, then $P \times_{G}F$ carries a twisted closed extension.
\end{proposition}

\begin{proof} According to the previous result, we have to find a one-form $\Theta$ such that $i_{b}^*\Theta=\theta_{b}$. 

We consider the one-form $(0, \theta)$ on $P \times F$. It is a closed $G$-invariant form. In order for $(0, \theta)$ to descend to a one-form on $P \times_{G} F$, the following relation has to be satisfied:
$$i_{\rho(v)}(0, \theta)=0,\quad \forall v \in \g,$$
where $\rho : \g \rightarrow \mathcal{X}(P \times F)$ is the infinitesimal action of the Lie algebra $\g$ on $P \times F$ with respect to the action of $G$ on $P \times F$ (given by $(p,f) \cdot g = (pg, g^{-1}f)$).

If we denote by $\rho': \g \rightarrow \mathcal{X}(P)$ the infinitesimal action of $\g$ on $P$, then $\rho=\rho'-\tilde\rho$. Hence, 
$i_{\rho(v)}(0, \theta)=-\theta(\tilde\rho(v))$ and the condition for $(0, \theta)$ to be basic becomes equivalent to the vanishing of $\theta(\tilde\rho(v))$. We consider then $\Theta$ on $P \times_{G}F$ such that $\pi^{*}\Theta=(0, \theta)$. Clearly, $i_{b}^*\Theta=\theta_{b}$ and the conclusion follows.
\end{proof}

\begin{remark} We shall see in the next section that two possible scenarios for $\theta(\tilde{\rho}(v))$ to vanish are when the action of $G$ is twisted Hamiltonian or when $G$ has at least one fixed point. 
\end{remark}

\section{Proof of  \ref{main}} 

The proof of the theorem follows the very same original arguments, meaning that we shall construct a non-degenerate two-form on the total space $P \times_G F$ by using a connection on $P$ and the momentum map of the twisted Hamiltonian action, as previously defined. However, this form must not be symplectic, but locally conformally symplectic. Therefore, the main difference from the original coupling form construction is that of finding a natural way to define a Lee form on $P \times_{G} F$ coinciding with the Lee forms of the fibres. For the sake of completeness, we shall give all the steps of the proof and compare to the results from the original setting (described for instance in \cite{ms}).
\smallskip

We need the following:

\begin{definition} A connection $\Gamma$ in an LCS fiber bundle is LCS if the parallel transport with respect to any curve $\gamma$, $P_{\gamma}: F_{\gamma_{0}} \rightarrow F_{\gamma_1}$, preserves the LCS forms of the fibers, meaning $P_{\gamma}^{*}\omega_{\gamma_1}=\omega_{\gamma_0}$ (which automatically implies $P_{\gamma}^{*}\theta_{\gamma_1}=\theta_{\gamma_0}$).
\end{definition}  

\begin{remark} Let $H$ be any connection on $P$. Then $H$ induces a connection $\mathcal{H}$ on $P \times_{G} F$ by the following:
$$\mathcal{H}_{[p,f]}=\{\pi_{*, (p,f)}v\,|\, v \in H_{p}\},$$
where $[p, f]$ denotes the image of $(p, f)$ under the projection $\pi : P \times F \rightarrow P \times_{G} F$. Moreover, $\mathcal{H}$ is an LCS connection and consequently the flow of every horizontal field on $P \times_{G} F$ with respect to $\mathcal{H}$ preserves the LCS form and the Lee form when restricted to the fibers. 
\end{remark}

\smallskip

The proof will be divided into six steps which we now explain:

\smallskip

\noindent{\bf{Step 1.} (Fixing a connection.)}
We denote by $B$ the quotient $P/G$ and by $\Ver$ the space of vertical fields of the fibration $F \rightarrow P \times_{G} F \rightarrow B$.  Let $\tilde\rho: g \rightarrow \mathcal{X}(F)$ be the infinitesimal action of $G$ on $F$. 

We consider on $P \times_G F$ a connection $\mathcal{H}$ induced by a fat connection $H$ on $P$ (we need fatness only at $\mu(F)$). The curvature of $\mathcal{H}$ is given by: $$R_{\mathcal{H}}: TB \times TB \rightarrow \Ver$$
$$R_{\mathcal{H}}(X, Y)=[X^*, Y^*]-[X, Y]^*,$$
where $X^*$ denotes the lift of $X$ with respect $\mathcal{H}$. We shall denote by $A$ the curvature form of the $\g$ - valued one-form $a$ associated to $H$. The relation between $A$ and $R_\mathcal{H}$ is: 
$$R_{\mathcal{H}}([p, f])(X, Y)=\pi_{*, (p, f)}(0, \tilde{\rho}(A_p(X, Y))(f)).$$

\remark As in \cite{v}, a vector field $X$ is twisted Hamiltonian if $i_{X}\omega$ is $d_{\theta}$ - exact. We denote the space of twisted Hamiltonian vector fields by $\Ham(F)$. Therefore, $\tilde{\rho}$ takes value in $\mathrm{Ham}(F)$ and the relation that connects the two curvatures tells us that $R_{\mathcal{H}}$ is a Hamiltonian vector field if restricted to a fiber $F_b$, meaning $R_{\mathcal{H}}: T_bB \times T_bB \rightarrow \mathrm{Ham}(F_b)$.

\smallskip

\noindent{\bf{Step 2.} ({Defining the Lee form.})} 
As in  \ref{clara}, the form $(0, \theta)$ on $P \times F$ descends to $P \times_{G}F$ if and only if $\theta(\tilde{\rho}(v))=0$ for any element $v$ in the Lie algebra $\g$. But this will be the case as a consequence of the twisted Hamiltonian action of $G$.
Indeed, $\tilde\rho(v)$ is a twisted Hamiltonian vector field on $F$, therefore there exists a smooth function $f$ on $F$ such that $i_{\tilde\rho(v)}\omega=d_{\theta}f$. By the Cartan formula we obtain that:
\begin{align*}
\mathcal{L}_{\tilde\rho(v)}\omega(X, Y)  =& i_{\tilde{\rho}(v)}d\omega(X, Y)+di_{\tilde\rho(v)}\omega(X, Y)\\
                              =&i_{\tilde\rho(v)}\theta \wedge \omega (X, Y) + d(d_{\theta}f)(X, Y)\\
                              =&\theta(\tilde\rho(v))\omega (X, Y) - \theta(X)\omega(\tilde\rho(v), Y)\\
                                 &+\theta(Y)\omega(\tilde\rho(v), X)+d(d_{\theta}f)(X, Y)\\
                              =&\theta(\tilde\rho(v))\omega (X, Y)-\theta \wedge d_\theta f (X, Y) + d(d_{\theta}f)(X, Y)\\
                              =&\theta(\tilde\rho(v))\omega(X, Y).
\end{align*}

However, $\mathcal{L}_{\tilde\rho(v)}\omega=0$, since the action of $G$ preserves $\omega$. The nondegeneracy of $\omega$ then implies  that $\theta(\tilde\rho(v))=0$, thus resulting in the condition of $(0, \theta)$ to be basic. 

Therefore, there exists a one-form $\Theta$ on $P \times_{G} F$ such that $(0, \theta)=\pi^{*}\Theta$. Moreover, $\Theta$ is  closed  and equals $\theta_b$ when restricted to $F_b$. The one-form $\Theta$ has also the property that $\Theta(\pi_{*}(w, 0))=0$, for any vector $w$ tangent to $P$. Thus, $(0, \theta)$ yields a closed one-form $\Theta$ on $P \times_{G} F$ which is $\theta_b$ on $F_b$ and vanishes on any horizontal space induced by a connection on $P$. 

$\Theta$ will be the Lee form of the LCS form we construct in the next steps.

\smallskip

\noindent{\bf{Step 3.} ({Defining the LCS form.}})} 
The LCS form is defined by intertwining the LCS structure of the fiber and the chosen connection on the principal bundle, being called \noindent{\em{coupling form}} in the original setting. We  use the same terminology and the definition is the following:

$$\omega_{[p, f]}(X, Y) = \begin{cases} \omega_{b}(X, Y) &\mbox{if } X \text{ and } Y \text{are vertical},\\
0 & \mbox{if } X \text{ is horizontal and }Y \text{ is vertical}, \\
\mu(f)(A_p(X', Y')) &\mbox{if } X \text{ and }Y \text{ horizontal}. \end{cases}$$

Here $X'$ and $Y'$ are the vectors in $T_{p}P$ such that $\pi_{*,(p,f)}(X')=X$ and $\pi_{*, (p,f)}(Y')=Y$.

\begin{remark} \label{rem1} Since $F$ is not a globally conformally symplectic manifold, $\theta$ is not an exact form and $H^0_{\theta}(F, \mathbb{R})$ vanishes (see \cite[Lemma 3.1]{hr1}). But $H^0_{\theta}(F, \mathbb{R})$ is precisely the obstruction for every Hamiltonian vector field to have a unique Hamiltonian associated function, as the following is an exact sequence:
$$0 \longrightarrow H^0_{\theta}(F, \mathbb{R}) \longrightarrow \mathcal{C}^{\infty}(F, \mathbb{R}) \longrightarrow \mathrm{Ham}(F) \longrightarrow 0.$$
Thus, the momentum map $\mu$ is defined by the $\mathcal{C}^{\infty}$ map  $\lambda: \Ham(F, \omega, \theta) \rightarrow \mathcal{C}^{\infty}(F)$ associating to a Hamiltonian vector field $X_f$ its unique Hamiltonian function $f$. Therefore, on horizontal vector fields, $\omega(x^*, y^*)=\mu(A(X', Y'))=\lambda(R_{\mathcal{H}}(x, y))$ where $x^*$ and $y^*$ are the horizontal lifts with respect to $\mathcal{H}$ of the vector fields $x$ and $y$ on $B$.
\end{remark}

\smallskip

\noindent{\bf{Step 4.} ({Verifying that $\omega$ is $d_{\Theta}$-closed.})}
We need to show $d_{\Theta}\omega(X, Y, Z)=0$ for all possible combinations of horizontal and vertical arguments.

\noindent{\bf 4.1.} By construction that $d_{\Theta}\omega(X, Y, Z)=0$, for vertical $X$, $Y$ and $Z$. 

\noindent{\bf 4.2.} Let $X$ and $Y$ be vertical and $Z$ is horizontal. Then $\Theta \wedge \omega (X, Y, Z)=0$, since  $\mathcal{H}$ is the $\omega$-orthogonal of the vertical bundle and $\Theta$ vanishes on $\mathcal{H}$. Thus, $d_{\Theta}\omega(X, Y, Z)=0$ is equivalent to $d\omega(X, Y, Z)=0$. But $\mathcal{H}$ is an LCS connection, and hence the flow of any horizontal vector field preserves $(\omega_b, \theta_b)$. This can be written as $(\mathcal{L}_{Z}\omega)(X, Y)=0$, for any horizontal field $Z$ and any vertical fields $X$ and $Y$.
As in \cite[Theorem 1.2.4]{gls}, the following computation shows that $d\omega(X, Y, Z)=0$:
\begin{equation}
d(i_{Z}\omega)(X, Y)=X(i_{Z}\omega(Y)) - Y(i_{Z}\omega(X))-i_{Z}\omega([X, Y])=0
\end{equation}
However, $0=(\mathcal{L}_{Z}\omega)(X, Y)=d(i_{Z}\omega)(X, Y)+i_{Z}d\omega(X, Y)$, hence $i_{Z}d\omega(X, Y)$ and therefore, also $d\omega(X, Y, Z)$ have to vanish.

\noindent{\bf 4.3.} Let now $X$ and $Y$ be horizontal fields and $Z$  vertical. The following relation  is essential for the proof:
\begin{equation}\label{hhv}
-d_{\Theta}\omega(Y, X)(Z)+d_{\Theta}\omega(Y, X, Z)=\omega([X, Y], Z).
\end{equation}
An equivalent form of the above is:
\begin{equation}\label{hhv2}
-d\omega(Y, X)(Z)+d\omega(Y, X, Z)+\omega(Y, X)\Theta(Z)-\Theta \wedge \omega(Y, X, Z)=\omega([X, Y], Z).
\end{equation}
Since horizontal and vertical fields are $\omega$ - orthogonal, $\Theta \wedge \omega(Y, X, Z)=\omega(Y, X)\Theta(Z)$, thus \eqref{hhv2} rewrites as 
$$-d\omega(Y, X)(Z)+d\omega(Y, X, Z)=\omega([X, Y], Z).$$ 
Although proved in \cite{gls} on page 6, we present the proof for the sake of completeness. By using the vanishing of the vertical - horizontal component of $\omega$, Cartan formula and the fact that the Lie bracket of two vector fields, one of which is vertical,  is again vertical, the following computation arises:
\begin{align*}
\mathcal{L}_{Z}\omega(Y, X)  =&(\mathcal{L}_Z\omega)(Y, X) + \omega([Z, Y], X) + \omega(Y, [Z, X])\\
                              =&(di_Z \omega + i_Z d\omega) (Y, X)\\
                              =& Y(i_Z \omega(X)) - X(i_Z \omega (Y)) - i_{Z} \omega ([Y, X]) + d\omega(Z, Y, X)\\
                              =&d\omega (Y, X, Z) - \omega ([X, Y], Z)
\end{align*}
Consequently,  relation \ref{hhv} holds.

Take now $X=x^*$ and $Y=y^*$ (where $x$ and $y$ are vector fields on $B$).   The $\omega$-orthogonality of $\mathcal{H}$ and $\Ver$ give:
$$i_{[a^*, b^*]}\omega(Z)=i_{[a^*, b^*]-[a, b]^*}\omega(Z)=i_{R_{\mathcal{H}}(a, b)}\omega(Z).$$
Using now \ref{rem1}, we further obtain that
$$i_{R_{\mathcal{H}}(a, b)}\omega(Z)=d_{\Theta}\lambda(R_{\mathcal{H}}(a^*, b^*))(Z)=d_{\Theta}\mu(A(a^*, b^*))=d_{\Theta}\omega(a^*, b^*).$$
Consequently, $i_{[a^*, b^*]}\omega(Z)=d_{\Theta}\omega(a^*, b^*)(Z)=-d_{\Theta}\omega(b^*, a^*)(Z)$. Relation (\ref{hhv}) now implies that $d_{\Theta}\omega(b^*, a^*, Z)=0$, hence $d_{\Theta}(X, Y, Z)=0$, for any $X$ and $Y$ horizontal and $Z$ vertical.

\noindent{\bf 4.4.} Finally, let $X, Y$ and $Z$ be horizontal vector  fields.  Since $\Theta$ vanishes on horizontal vector fields, $d_{\Theta}\omega(X, Y, Z)=0$ is equivalent to $d\omega(X, Y, Z)=0$. The proof is the same as in \cite{ms} on page 225, and hence we only indicate a sketch of it.

 Denote by $X^{v}$ the vertical component of the vector field $X$. Then the following relation is straightforward:
$$[[x^*, y^*], z^*]^v=R_{\mathcal{H}}([x, y], z)-[z^*, R_{\mathcal{H}}(x, y)].$$
But $R_{\mathcal{H}}([x, y], z)$ and $[z^*, R_{\mathcal{H}}(x, y)]$ are fiberwise Hamiltonian vector fields. It is easy to show that the fiberwise Hamiltonian function for $[z^*, R_{\mathcal{H}}(x, y)]$ is $\mathcal{L}_{z^*}\omega(x^*, y^*)$ (because $\omega(x^*, y^*)$ is a fiberwise Hamiltonian function for $R_{\mathcal{H}}(x, y)$). So $[[x^*, y^*], z^*]^v$ is a fiberwise Hamiltonian vector field, and its Hamiltonian function is $\omega([x^*, y^*],z^*)-\mathcal{L}_{z^*}\omega(x^*, y^*)$ (again, we used the fact that the vertical-horizontal component of $\omega$ vanishes). By summing over cyclic permutations and using the Jacobi identity, one gets that $\sum_{x,y,z}[[x^*, y^*], z^*]^v=0$, so its Hamiltonian function has to be 0 because of its uniqueness. What we get is:
$$\sum_{x,y,z}\omega([x^*, y^*],z^*)-\sum_{x,y,z}\mathcal{L}_{z^*}\omega(x^*, y^*)=0$$ and this is precisely the expression of $d\omega(x^*,y^*,z^*)$, which ends the proof.

\smallskip

\noindent{\bf{Step 5.} ({The non-degeneracy of $\omega$.})} 
Choosing a fat connection on $P$ at $\mathrm{Im}\mu$ will assure the non-degeneracy of $\omega$ on the horizontal space, which is the only thing we need to know, since on the vertical bundle it is already non-degenerate, thus ending the construction.

\begin{remark} Circle LCS actions provide more easily examples of LCS manifolds constructed as in \ref{main}. The reason is that principal $S^1$-bundles over some manifold $M$ are parametrized by $H^2(M, \Z)$. Moreover, if one chooses a symplectic manifold $M$ whose symplectic form $\omega$ has integral cohomology class, then the corresponding $S^1$-bundle over $M$ is endowed with a fat connection, namely the 1-form $\alpha$ satisfying $\pi^*\omega=d\alpha$, where $\pi$ is the natural projection $\pi: P \rightarrow M$. The non-degeneracy of $\omega$ implies the fatness of $\alpha$. The 2-form $d\alpha$ represents also the curvature of $\alpha$, hence the coupling form of the bundle $P \times_{S^1} F$ rewrites on the horizontal part as $\Omega(X^*, Y^*)=\mu(\pi^*\omega(X', Y'))$, where $X^*$ and $X'$ denote the lift                                                                                               of the vector field $X$ on $P \times_{S^1}F$ and respectively on $P$.
\end{remark}

\section {Twisted Hamiltonian actions}

The above results show the need to  find twisted Hamiltonian actions on LCS manifolds. In this section we give several criterions for the existence of such actions and conclude with several examples. 

We already used and proved the following result within the proof of  \ref{main}:

\begin{proposition} Let $(M, \omega, \theta)$ be an LCS manifold and let $G$ be a Lie group acting on $M$. If the action of $G$ is twisted Hamiltonian, then necessarily $\theta(\rho(v))=0$ for any fundamental vector field $\rho(v)$ of the action.
\end{proposition}

\smallskip

The condition $\theta(\rho(v))=0$ is also sufficient in the following case:

\begin{proposition}\label{ant} Let $(M, \omega, \theta)$ be an LCS manifold with $H_{\theta}^{1}(M, \mathbb{R})=0$. Then the action of a group $G$ is twisted Hamiltonian if and only if $\theta(\rho(v))=0$. 
\end{proposition}

\begin{proof} The direct implication is clear. For the converse, we prove that $i_{\rho(v)}\omega$ is $d_{\theta}$-closed. By using the Cartan formula and the $G$ - invariance of $\omega$, we get:
\begin{align*} d_{\theta}i_{\rho(v)}\omega&=di_{\rho(v)}\omega - \theta \wedge i_{\rho(v)}\omega\\
&=\mathcal{L}_{\rho(v)}\omega-i_{\rho(v)}d\omega-\theta \wedge i_{\rho(v)}\omega\\
&=-i_{\rho(v)}\theta\wedge\omega-\theta\wedge i_{\rho(v)}\omega\\
&=0.
\end{align*}

Since $H^{1}_{\theta}(M, \mathbb{R})$ vanishes, any $d_{\theta}$-closed form is $d_{\theta}$-exact, hence the action is twisted Hamiltonian.
\end{proof}

The next result adds a condition for an action to be twisted Hamiltonian in terms of the Lie algebra of the group:

\begin{proposition} If $G$ is a Lie group with perfect Lie algebra $\mathfrak{g}$ acting on the LCS manifold $(M, \omega, \theta)$ by preserving $\omega$ and $\theta(\rho(v))=0$, for any $v$ in $\mathfrak{g}$, then $G$ acts twisted Hamiltonian.
\end{proposition}

\begin{proof} As in \ref{ant}, we get that $i_{\rho(v)}\omega$ is $d_{\theta}$-closed.
Hence we can define a map $\lambda: \mathfrak{g} \rightarrow H^1_{\theta}(M)$ by 
$$v \mapsto [i_{\rho(v)}\omega].$$
We prove next  that whenever $X$ and $Y$ have $\omega$-preserving flows and $\theta(X)=\theta(Y)=0$, then $[X, Y]$ is a twisted Hamiltonian vector field (i.e. $i_{[X, Y]}\omega$ is $d_{\theta}$ -exact). Indeed, by using the Lie derivative properties, we get:
\begin{align*}
i_{\mathcal{L}_XY}\omega &=\mathcal{L}_X(i_{Y}\omega)-i_{Y}\mathcal{L}_X\omega\\
&=di_{X}(i_{Y}\omega)+i_Xd(i_Y\omega)+0\\
&=-d_{\theta}(\omega(X, Y))
\end{align*}
By the easily verifiable relation $[\rho(v), \rho(w)]=-\rho([v, w])$ and by taking $\rho(v)$ and $\rho(w)$ as $X$ and $Y$ above, one obtains that $[\mathfrak{g}, \mathfrak{g}]$ is in the kernel of $\lambda$. Since $\mathfrak{g}$ is perfect, $\lambda$ vanishes identically, hence the action is twisted Hamiltonian.
\end{proof}

\begin{remark}   The condition $\theta(\rho(v))=0$ plays a very important role in finding twisted actions. By the Cartan formula and the fact that $\theta$ is closed and $G$-invariant, we get that $0=\mathcal{L}_{\rho(v)}\theta=d(\theta(\rho(v))).$ Therefore, $\theta(\rho(v))$ is always a constant (we always assume the manifolds we work with are connected!). We observe that whenever the action of $G$ has a fixed point ($m$ is fixed if $g \cdot m = m$, for any $g \in G$), $\theta(\rho(v))$ vanishes, since the fixed points provide zeros of $\rho(v)$. 

Another case that triggers the desired equality is when $\theta$ has vanishing points or in general when $M$ is a {\em second kind} LCS manifold. This is a notion introduced in \cite{v}: first and second kind LCS manifold are defined in terms of the Lee homomorphism (which associates to any vector field $X$ whose flow preserves the LCS form $\omega$ the constant $\theta(X)$). A second kind LCS manifold is then defined as having the Lee homomorphism trivial, thus satisfying the condition $\theta(\rho(v))=0$. 
\end{remark}

\smallskip

The next criterion deals with $d_{\theta}$ - exact LCS forms.

\begin{lemma}\label{fixed}  Let $(F,\omega,\theta)$ be a connected LCS manifold with $d_\theta$-exact LCS form, $\omega=d_{\theta}\eta$, and let $G$ act on $F$ by preserving $\eta$ and $\theta$. If $G$ has at least one fixed point, then the action of $G$ is twisted Hamiltonian.
\end{lemma}

\begin{proof}
We want to show that $i_{\rho(v)}\omega$ is $d_{\theta}$-exact. The condition $\theta(\rho(v))=0$ is satisfied, since the hypothesis fits one of the situations described in the previous remark. Using this and the Cartan formula again, we obtain:
\begin{align*}
0=\mathcal{L}_{\rho(v)}\eta &=di_{\rho(v)}\eta+i_{\rho(v)}d\eta\\
                                                  &=d \eta(\rho(v))+i_{\rho(v)}(\omega+\theta\wedge\eta)\\
                                                  &=d \eta(\rho(v)) + i_{\rho(v)}\omega - \eta(\rho(v))\theta.
\end{align*}
This yields $i_{\rho(v)}\omega=-d_{\theta}\eta(\rho(v))$ and the action is indeed twisted Hamiltonian. 
\end{proof}

\begin{remark}
For the above twisted Hamiltonian action, the twisted momentum map is given by
$$\mu:F \rightarrow \g^{*}, \quad 
\mu(f)(v)=-\eta(\rho(v))(f).$$
\end{remark}

\begin{remark}
 Let $M$ be a contact manifold and $G$ a Lie group acting by preserving some contact form $\alpha$. If we let $G$ act trivially on $S^1$, then the action of $G$ on $S^1 \times M$ is twisted Hamiltonian with respect to the lcs form $d_{d\mathrm{vol}}\alpha$, where $d\mathrm{vol}$ is a volume form on $S^1$. The proof follows the arguments above, because the condition of existence of at least one fixed point can be dropped. Indeed, this condition was important only for achieving the vanishing of $d\mathrm{vol}(\rho(v))=0$. But this is also the case, since any fundamental vector field $\rho(v)$ is tangent to $M$ only, hence the $S^1$-component vanishes. Note that there are many examples of LCS manifolds of the form $S^1\times M$ with $M$ contact (in particular, Sasakian, \cite{ov}). See also \ref{conbun}.
\end{remark}

\smallskip

Another approach towards finding twisted Hamiltonian actions is moving from the LCS manifold to one of its symplectic coverings. Unfortunately, the action of a group does not in general lift to an action on a covering space. However, if $G$ is the Lie group acting on $M$ and $\tilde{M}$ is one of its coverings, one can always define on $\tilde{M}$ an action of the connected component of the identity of the universal covering of $G$, which we shall denote henceforth by $\tilde{G_0}$. For more details, see \cite{mo}.

Intuitively, a twisted Hamiltonian action lifts to a Hamiltonian one and the reversed scenario also should work. Indeed, this will be the case under some minor extra requirements. The following two results concern this situation.

\begin{proposition}\label{propa} Let $(M, \omega, \theta)$ be an LCS manifold and $\pi: \tilde{M} \rightarrow M$ one of its symplectic coverings. If $G$ is a Lie group acting twisted Hamiltonian on $M$, then $\tilde{G_0}$ acts Hamiltonian on $\tilde{M}$.
\end{proposition}  

\begin{proof} 


Let $\pi^*\theta=dh$ on the covering $\tilde{M}$. The two-form $\Omega=e^{-h}\pi^*\omega$ is symplectic. We denote by $\tilde{\rho}(v)$ the fundamental vector field of $v$ with respect to the action of $\tilde{G_0}$ on $\tilde{M}$. One can easily show that $\pi_{*}\tilde{\rho}(v)=\rho(v)$, hence $i_{\tilde{\rho}(v)}\Omega=e^{-h}i_{\tilde{\rho}(v)}\pi^{*}\omega=e^{-h}\pi^{*}(i_{\rho(v)}\omega)$. However, $G$ acts twisted Hamiltonian on $M$, therefore there exists $s$ such that $i_{\rho(v)}\omega=d_{\theta}s$, meaning that 
$$i_{\tilde{\rho}(v)}\Omega=e^{-h}\pi^{*}(d_{\theta}s)=e^{-h}\pi^{*}(d_{dh}s)=d(e^{-h}\pi^{*}s),$$ 
and the action is thus Hamiltonian. 
\end{proof}


\smallskip

We present next a partial converse of \ref{propa} and following the definitions in \cite{gopp}, we recall:

\begin{definition} Let $\tilde{M}$ be the universal covering of an LCS manifold and $\chi: \pi_1(M) \rightarrow R^{+}$ the character of $M$. Then $\bar{M}:=\tilde{M}/\Ker(\chi)$ is called the minimal covering of $M$.
\end{definition}

\begin{remark}\label{abel} $\mathrm{Deck}(\bar{M}/M)$ is an abelian group, since it is isomorphic to the quotient $\pi_1(M)/\Ker(\chi)$, which is a subgroup of $\R^+$. Moreover, $\mathrm{Deck}(\bar{M}/M)$ acts only by proper homotheties, since factorizing to $\Ker \chi$ means factorizing to the isometries of $\pi_1(M)$.
\end{remark}

\begin{theorem}\label{minimal} 
Let $\pi: \bar{M} \rightarrow M$ be the minimal covering of the LCS manifold $(M, \omega, \theta)$ and let $\Omega$ be the symplectic form of $\bar{M}$. Let $G$ be a Lie group acting on $M$ by preserving $\omega$. Then $G$ acts twisted Hamiltonian on $M$ if and only $\tilde{G_0}$ acts Hamiltonian on $\bar{M}$.
\end{theorem}

\begin{proof} 
The direct implication is exactly \ref{propa}, so we are left only with the converse.
Let us denote by $c_{\gamma}$ all the scalars of the homotheties in $\mathrm{Deck}(\bar{M}/M)$ (meaning $\gamma^*\Omega=c_{\gamma}\Omega$, for any $\gamma$ in $\mathrm{Deck}(\bar{M}/M)$). By \ref{abel}, we have $c_{\gamma} \neq 1$. The symplectic form $\Omega$ is equal to $e^{-h}\pi^*\omega$, where $\pi^*\theta =dh$. It is a straightforward computation to show that $c_{\gamma}=e^{-\gamma^*h+h}$.

The Hamiltonian action of $\tilde{G_0}$ on $\bar{M}$ implies that for any $v$ in the Lie algebra $\g$: 
$$i_{\tilde{\rho}(v)}\Omega=df.$$
One can show that $i_{\tilde{\rho}(v)}\Omega$ is an automorphic 1-form, so $df$ is automorphic, whence we get that $\gamma^*f-c_{\gamma}f$ is constant. We shall denote henceforth this constant depending on $\gamma$ by $a_{\gamma}$.

Next, we prove that 
\begin{equation}\label{constantak}
k(\gamma): =\frac{a_{\gamma}}{1-c_{\gamma}} 
\end{equation}
is a constant not depending on $\gamma$ in $\mathrm{Deck}(\bar{M}/M)$.

In order to prove it, we need to see how the constant $a_{\gamma}$ behaves for the composition of two elements $\gamma_1$ and $\gamma$. By pulling back the equality $\gamma^{*}f-c_{\gamma}f=a_{\gamma}$ with $\gamma_{1}$, we get $\gamma_1^*\gamma^*f - c_{\gamma}\gamma_1^*f=a_{\gamma}$, which is equivalent to
$$(\gamma \circ \gamma_1)^*f - c_{\gamma_1}c_{\gamma}f+c_{\gamma_1}c_{\gamma}f - c_{\gamma}\gamma_1^*f=a_{\gamma}.$$
Consequently, 
$$a_{\gamma \circ \gamma_1} - c_{\gamma}a_{\gamma_1} =a_{\gamma}$$ for any $\gamma_1$, $\gamma$ in $\mathrm{Deck}(\bar{M}/M)$. By changing $\gamma$ with $\gamma_1$ in the equality above, we get:
$$a_{\gamma_1 \circ \gamma}-c_{\gamma_1}a_{\gamma}=a_{\gamma_1}.$$
However, $\mathrm{Deck}(\bar{M}/M)$ is abelian by \ref{abel}, hence $a_{\gamma_1 \circ \gamma}=a_{\gamma \circ \gamma_1}$. Combining the previous two equalities, we get
$$c_{\gamma_1}a_{\gamma}+a_{\gamma_1}=c_{\gamma}a_{\gamma_1}+a_{\gamma},$$
so $a_{\gamma_1}(1-c_\gamma)=a_{\gamma}(1-c_{\gamma_1})$, for any $\gamma$ and $\gamma_1$ in $\mathrm{Deck}(\bar{M}/M)$, whence we conclude that $k(\gamma)=k$ is constant. 

We now explain why the action of $G$ is twisted Hamiltonian. We first notice that $e^h(f-k)$ is $\mathrm{Deck}(\bar{M}/M)$-invariant, and hence descends to $M$. Indeed, 
\begin{equation*}
\begin{split}
\gamma^{*}e^h(f-k)&=e^{\gamma^*h}(\gamma^*f-k)=\tfrac{e^{h}}{c_{\gamma}}(a_\gamma+c_{\gamma}f-k)\\
&=\tfrac{e^{h}}{c_{\gamma}}(k(1-c_{\gamma})+c_\gamma f-k)=e^h(f-k).
\end{split}
\end{equation*}
Therefore, there exists some function $r$ on $M$ such that $e^h(f-k)=\pi^{*}r$. This further yields:
$$df=d(e^{-h}\pi^{*}r)=e^{-h}d_{dh}\pi^*r=e^{-h}d_{\pi^*\theta}\pi^{*}r=e^{-h}\pi^{*}(d_{\theta}r).$$
Nevertheless, 
$$df=i_{\tilde{\rho}(v)}\Omega=e^{-h}i_{\tilde{\rho}(v)}\pi^*\omega=e^{-h}\pi^{*}(i_{\rho(v)}\omega),$$
 and hence we obtain  
$$\pi^*(i_{\rho(v)}\omega)=e^hdf=\pi^{*}(d_{\theta}r),$$
 meaning that the action is indeed twisted Hamiltonian.
\end{proof}

\begin{remark}\label{rem} If the universal covering $\tilde{M}$ coincides with the minimal covering, the requirement for the action of $G$ on $\tilde{M}$ to be Hamiltonian simplifies to preserving the symplectic form, since $\tilde{M}$ is simply connected. 
\end{remark}

\begin{remark}\label{inoue} \ref{minimal} does not work in general for any covering space different from the minimal one. Let us take $\tilde{M}$ a symplectic covering of $M$ and $G$ a Lie group acting by preserving $\omega$. If the lift                                                                                               of this action to $\tilde{M}$, namely the action of the covering $\tilde{G_0}$, is Hamiltonian with respect to the symplectic form $e^{-h}\pi^*\omega$, then $G$ is twisted Hamiltonian if and only if one can always choose on $\tilde{M}$ an automorphic Hamiltonian function for any fundamental vector field, as the proof of \ref{minimal} unravels. Fortunately, this can always be done on the minimal covering, since the deck transformations are not isometries and the function introduced at \eqref{constantak} is constant. We shall give at the end of the section an example consisting of a Hamiltonian action on the universal cover of an LCS manifold, which does not come from a twisted Hamiltonian action.
\end{remark}

\begin{example} 
 We construct a twisted  Hamiltonian action on diagonal Hopf manifolds. We regard $S^{2n-1}$ as embedded in $\C^n \setminus\{0\}$ by:
$$S^{2n-1}=\{(z_1, z_2, \ldots, z_{n})\,|\, \vert z_{1} \vert^2 + \cdots + \vert z_{n} \vert^2=1\}.$$

There are many contact structures on  $S^{2n-1}$ (see, {\em e.g.} \cite{ko}). Namely, every collection of real numbers (here called {\em weights}) $a_{1}, \ldots, a_{n}$, satisfying:
$$0<a_1\leq a_2 \ldots \leq a_n$$
provides us with a contact form 
$$\eta_{a_1, \ldots, a_n}=\tfrac{1}{\sum_{i}a_i\vert z_i \vert^2}\eta_0$$ which is a deformation of the standard one  of $S^{2n-1}$: $\eta_0 = \sum_{i}y_idx_i - x_idy_i$. 

These contact forms will further provide a family of LCS forms on $S^{1} \times S^{2n-1}$ defined by $$\omega_{a_1, \ldots ,a_{n}}=d_{d\mathrm{vol}}\eta_{a_1, \ldots, a_n}$$ where $d\mathrm{vol}$ is a volume form on $S^1$. The universal covering $\R \times S^{2n-1}$ carries the exact symplectic form $d(e^t\eta_{a_1, \ldots, a_n})$. By dividing it with $e^t$ we obtain the two-form $\frac{d(e^t\eta_{a_1, \ldots, a_n})}{e^t}$ which descends to $\omega_{a_1, \ldots, a_n}$ on $S^1 \times S^{2n-1}$.

Let $T^k$ be the $k$ dimensional torus, with $k \in \{1, 2, \ldots n\}$. One can define the following action of $T^{k}$ on $S^1 \times S^{2n-1}$:
$$(u_1, u_2, \ldots , u_{k}) \cdot (s, (z_1, z_2, \ldots, z_{n}))=(s, (u_1z_1, \ldots, u_{k}z_k,z_{k+1} \ldots, z_n)).$$ 
 This action preserves $d\mathrm{vol}$ and also the contact form $\eta_{a_1, \ldots, a_n}$ for any $n$-tuple $a_1, \ldots, a_n$, hence the LCS form is $T^k$ - invariant. Moreover, the condition $$d\mathrm{vol}(\rho(v))=0$$ is satisfied for any fundamental vector field since its $S^1$-component is 0. From the proof of \ref{fixed} (we need only the vanishing of $d\mathrm{vol}(\rho(v))$, not necessarily a fixed point of the action), the above action is twisted Hamiltonian.

Another approach for this example would be using the isomorphism between $\R \times S^{2n-1}$ and $\C^n\setminus \{0\}$ given by:
$$\psi_{a_1, \ldots, a_n}(t, (z_1, \dots, z_n))=(e^{-ta_1}z_1, \ldots, e^{-ta_n}z_n).$$
Thus, we recover $S^1 \times S^{2n-1}$ as $\C^n\setminus\{0\}/((z_1, \ldots, z_n) \mapsto(e^{a_1}z_1, \ldots e^{a_n}z_n))$. The induced action of $T^k$ on 
$\C^n\setminus\{0\}$ is $$(u_1, \ldots, u_k)\cdot (z_1, \ldots, z_k)=(u_1z_1, \ldots, u_{k}z_k,z_{k+1} \ldots, z_n)$$ and it is Hamiltonian with respect to the symplectic form 
$$\omega=(\psi_{a_1, \ldots, a_n})_*(d(e^t\eta_{a_1, \ldots, a_n})).$$ The hypothesis conditions in \ref{rem} are satisfied, therefore the action is twisted Hamiltonian. 

The motivation for considering this example  actually comes from complex geometry. However, we do not need any complex structure for our purposes, but since the action we consider happens to be holomorphic, our setting fits the situations described in \cite{gop}. Another way to see the action is indeed twisted Hamiltonian is the following: Hopf manifolds are Vaisman, as proven in \cite{ko}. Moreover, the torus $T^k$ acts by Vaisman automorphisms as defined in \cite{gop}. Theorem 5.15 in \cite{gop} proves that every group acting by Vaisman automorphisms has a twisted Hamiltonian action.
\end{example}

\begin{remark}\label{conbun} 
  The above example is related to contact fibre bundles in the following way.

Let $P$ be a principal $T^{k}$ - bundle. The product $P \times_{T^{k}}(S^1 \times S^{2n-1})$ is  equal to $(P \times_{T^{k}}S^{2n-1}) \times S^1$, since the action  of $T^k$ is trivial on $S^1$. Nevertheless, $T^{k}$ preserves $\eta_{a_1, \ldots, a_n}$, hence $T^{k}$ acts on $S^{2n-1}$ by contactomorphisms and the contact momentum map is given by $v \mapsto \eta_{a_1, \ldots, a_n}(\rho(v))$. Thus, the twisted momentum map on $S^1 \times S^{2n-1}$ and the contact momentum map on $S^{2n-1}$ have the same image in the dual of the Lie algebra of the torus,  $(\mathbb{R}^{k})^*$. Choosing a connection in $P$ which is fat on the points in the image of the momentum map will produce a contact form on $P \times_{T^{k}} S^{2n-1}$, according to  \cite[Theorem 3.3]{l}, and hence an LCS structure on $(P \times_{T^{k}} S^{2n-1}) \times S^1$.

\end{remark}

\begin{example} {For any manifold $M$, $T^*M$ has a locally conformally symplectic structure given by the Lee form $\pi^{*}\alpha$ and the LCS form $d_{\pi^*\alpha}\theta$, where $\alpha$ is any closed form on $M$, $\theta$ is the Liouville form on $T^*M$ and $\pi : T^*M \rightarrow M$ is the natural projection, \cite{haller}.}

Let  $G$ be a {\em compact} Lie group acting on $M$ with at least one fixed point. We lift the action of $G$ to $T^{*}M$. This action will preserve $\theta$. As $G$ is compact, we can average any closed 1-form on $M$ and obtain a $G$-invariant closed one-form $\alpha$. The action of $G$ on $T^{*}M$ will also preserve $\pi^{*}\alpha$ and since $G$ acts with at least one fixed point, any $G$-invariant one-form $\beta$ evaluated in a fixed point $p$ will provide a fixed point $\beta_p$ for the action of $G$ on $T^*M$, recovering thus the setting from  \ref{fixed}. On the other hand, $(T^*M, d_{\pi^*\alpha}\theta, \pi^{*}\alpha)$ is not globally conformal symplectic unless $\alpha$ is an exact form on $M$, and hence non-trivial examples are obtained for avery $M$ with non-trivial $H^1(M,\R)$.
\end{example}

\smallskip

\begin{example}\label{exinoue}
We present an $S^1$-action on the Inoue surface $S^+_{N, p, q,r,t}$ which can also be found in \cite{k}, that preserves the LCS form given in  \cite{t}, is not twisted Hamiltonian, but its lift to an action of $\R$ on the universal covering is Hamiltonian. This shows the relevance of the minimal covering in \ref{minimal} and relates directly to \ref{inoue}.

We recall the definition of $S^+_{N, p, q,r,t}$, see \cite{i}. Let $p, q, r$ be integers with  $r\neq 0$ and let $t$ be a real number. Let  $N\in \mathrm{SL}(2,\Z)$, with eigenvalues  $\alpha$ and $\tfrac{1}{\alpha}$, and let $(a_1, a_2)$ and $(b_1, b_2)$ be two corresponding eigenvectors. We denote by $(c_1, c_2)$ the solution of the equation:
$$(c_1, c_2)=(c_1, c_2) \cdot N^{ T}+(e_1, e_2)+\frac{b_1a_2-b_2a_1}{r}(p, q),$$
where 
$$e_i=\tfrac{1}{2}n_{i1}(n_{i1}-1)a_1b_1+\tfrac{1}{2}n_{i2}(n_{i2}-1)a_2b_2+n_{i1}n_{i2}b_1a_2, \quad i=1,2.$$
Let $\H$ denote the upper half plane $\{w\in\C\,;\, \Im w>0\}$ and call  $G^+$  the group of automorphisms of $\mathbb{H} \times \mathbb{C}$ generated by:
$$g_{0}(w, z)=(\alpha w, z+t)$$
$$g_{i}(w,z)=(w+a_i, z+b_iw+c_i)$$
$$g_3(w, z)=(w, z+ \tfrac{b_1a_2-b_2a_1}{r}).$$
The $S^+_{N, p, q, r, t}$ Inoue surface is defined as $(\mathbb{H} \times \mathbb{C})/G^+$. According to \cite{t}, the LCS form carried by this complex surface is written on the cover $\mathbb{H} \times \mathbb{C}$ as:
$$\tilde{\omega}=-\mathrm{i}\left(\frac{1+(z_2)^2}{(w_2)^2}dw \wedge d\bar{w} - \frac{z_2}{w_2}(dw \wedge d\bar{z} + dz \wedge d\bar{w}) + dz \wedge d\bar{z}\right).$$
The Lee form is written on $\H \times \C$ as $d \mathrm{ln} w_2$ and it is easy to see that indeed, $d\tilde{\omega}=d \mathrm{ln} w_2 \wedge \tilde{\omega}$. We denote by $s$ the real number $\frac{b_1a_2-b_2a_1}{r}$ and regard the circle $S^1$ as the quotient $\R/\sim$, where $x\sim x+s$.

We define an $S^1$-action on $S^+_{N, p, q, r, t}$ by:
$$[x] \cdot \widehat{(w, z)}=\widehat{(w, z+x)}.$$ 
Direct computations show that this is indeed a well defined action, preserving the LCS structure. This $S^1$ actions  lifts                                                                                               to an action of $\R$ on $\H \times \C$: 
$$x \cdot (w, z)=(w, z+x),$$
which clearly leaves the symplectic form $\tfrac{1}{w_2}\tilde{\omega}$ invariant,and  hence it is Hamiltonian on the universal cover of $S^+_{N, p, q, r, t}$. However, the $S^1$ action is not twisted Hamiltonian. Indeed, if it was twisted Hamiltonian, then, according to \ref{inoue}, one could  choose a Hamiltonian function on $\H \times \C$ which is automorphic. Then each Hamiltonian function is left invariant by all the isometries in $\mathrm{Deck}((\H \times \C) /S^+_{N, p, q, r, t})$. But the fundamental vector field corresponding to the element $1$ from the Lie algebra (equal to $\R$) is $\partial_z + \partial_{\bar{z}}$ and its Hamiltonian function is $-2\frac{z_2}{w_2}$ and it is not preserved by the isometries $g_2$ and $g_3$ from $G^+$, which is the deck group of the covering, contradiction.
\end{example}

In particular, this proves:
\begin{corollary} $H^1_{\theta}(S^+_{N, p, q, r, t})$ is nontrivial.
\end{corollary}
\begin{proof} We consider the action described in \ref{exinoue} and use the same notations. We first notice that $\theta(\rho(v))=0$ for any real number $v$. This happens because $\pi^*\theta(v(\partial_z+\partial_{\bar{z}}))=d\mathrm{ln}w_2(v(\partial_z+\partial_{\bar{z}}))=0$ and combined with the fact that $\omega$ is invariant under the action of $S^1$, one obtains that $i_{\theta}\rho(v)$ is $d_{\theta}$-closed. If $H^1_{\theta}(S^+_{N, p, q, r, t})$ vanished, then the action would be twisted Hamiltonian, which is not the case. 
\end{proof}

\section{Compatibility of the bundle construction with LCS reduction}

In this section we study compatibility properties with respect to locally conformally symplectic reduction. For that we need to consider only contexts when {\em the same  group $G$ acts on the fiber and on the total space of the fibration in a compatible way}. The following results show that a favorable case for this to happen is when $G$ is abelian.

\begin{proposition}
 If $G$ is an abelian Lie group acting twisted Hamiltonian on the LCS manifold $(F, \omega, \theta)$, then the Hamiltonian function of every fundamental vector field is $G$-invariant.
\end{proposition}

\begin{proof}
 The commutativity of $G$ assures that every fundamental vector field $\rho(v)$ is $G$-invariant. Let $h$ be the unique Hamiltonian function for $\rho(v)$. Then $i_{\rho(v)}\omega=d_\theta h$. Moreover, since $\omega$, $\theta$ and $\rho(v)$ are $G$-invariant, $d_{\theta}g^*h=d_\theta h$ for any element $g$ in $G$. However, $H^0_{\theta}(F)$ vanishes (see \ref{rem1}), as $\theta$ is not exact, so $g^*h=h$.
\end{proof}

\begin{proposition}\label{fundamental} If $G$ is an abelian Lie group acting twisted Hamiltonian on $F$ and $P$ is a $G$-principal bundle with a fat connection, then $G$ acts twisted Hamiltonian on $P \times_G F$ with respect to the coupling form constructed in \ref{main}.
\end{proposition}

\begin{proof} Define the action of $G$ on $P \times_G F$ by:
$$g \cdot [p, f]=[p, g\cdot f].$$
This is well defined as a consequence of the group $G$ being abelian. In order to check this action is indeed twisted Hamiltonian, we observe that the fundamental vector fields $\tilde{\rho}(v)$ (for any $v$ in $\mathfrak{g}$) are vertical and fiberwise twisted Hamiltonian, since $\tilde{\rho}(v)=\pi_{*}(0, \rho(v))$, where $\pi: P \times F \rightarrow P \times_G F$ is the projection. 

If $\Omega$ is the coupling form and $\Theta$ is its corresponding Lee form, we shall prove that $i_{\tilde{\rho}(v)}\Omega$ is $d_{\Theta}$-exact.
Let $h$ be the Hamiltonian function of $\rho(v)$. According to \ref{fundamental}, $h$ is $G$-invariant. We denote by $\tilde{H}: P \times F \rightarrow \R$ the function $\tilde{H}(p, f)=h(f)$. Since $\tilde{H}$ is also $G$-invariant, it descends to a function $H: P \times_G F \rightarrow \R$, such that $H([p, f])=h(f)$. Moreover, $H$ has the property that restricted to any fiber $F_b$ is the Hamiltonian function associated to $\rho(v)$. We prove the following: 
\begin{equation}\label{eee}
i_{\tilde{\rho}(v)}\Omega=d_{\Theta}H.
\end{equation}
Equality \eqref{eee} is obviously true on vertical vector fields, since its restriction to the fiber $F_b$ reduces to the action of $G$ on $F$ being twisted Hamiltonian. We have to prove next that \eqref{eee} also holds on horizontal vector fields.  

Let $V^*$ be the lift to $P \times_G F$ of the vector field $V$ on $P/G$. As shown in the proof of \ref{main},  $[V^*, \tilde{\rho}(v)]$ is a fiberwise Hamiltonian vector field, whose Hamiltonian function is $\mathcal{L}_{V^*}H$.  However, $[V^*, \tilde{\rho}(v)]$ is $\pi_*[V^{**}, \rho(v)]$, where $V^{**}$ is the lift of $V$ to $P$, thus the Lie bracket vanishes and therefore, $\mathcal{L}_{V^*}H=0$. Then the following holds:
$$d_{\Theta}H(V^*)=dH(V^*) - H\Theta(V^*)=0=i_{\tilde{\rho}(v)}\Omega(V^*),$$
as $\Theta$  vanishes on basic vector fields. Thus, we have proved that \eqref{eee} is indeed satisfied. 

According to our definition of twisted Hamiltonian action, we still have to show that for any $g$ in $G$:
\begin{equation}\label{invar}
g^*\Omega=\Omega. 
\end{equation}
For $X$ and $Y$ vertical vector fields, \eqref{invar} means the $G$-invariance of $\omega$, which is granted. For $X$ horizontal and $Y$ vertical, $\Omega(X, Y)$ vanishes and as the action of $G$ on $P \times_G F$ preserves the horizontality and verticality of vector fields, \eqref{invar} holds whenever $X$ is horizontal and $Y$ is vertical. We are left with verifying that \eqref{invar} is true also when $X$ and $Y$ are both horizontal. By using the fact that for any $g$ in $G$ and $X$ horizontal, $g_{*}X=X$ and the $G$-invariance of the momentum map (as a consequence of the group being abelian), one gets: $$\mu(A(X, Y))=\mu(A(g_*X, g_*Y)),$$
which is the definition of $\Omega$ on horizontal vector fields and concludes the $G$-invariance of $\Omega$.
\end{proof}

\medskip

Let now $\tilde{\mu}$ be the momentum map of the action of $G$ on $P \times_G F$. Since $G$ is abelian, both $\mu$ and $\tilde{\mu}$ are $G$-invariant. The momentum map $\tilde{\mu}$ can be expressed as:
$$\tilde{\mu}([p, f])(v)=H([p, f])=h(f)=\mu(f)(v).$$
So, the relation between $\tilde{\mu}$ and $\mu$ is the following:
\begin{equation}\label{moment}
\tilde{\mu}([p, f])=\mu(f).
\end{equation}

Usually, the reduction process is performed at 0. In the present case, the reduction can be performed at any regular value, since the momentum maps are $G$-invariant. For simplicity, however, we only consider reduction at 0.  

As a consequence of \eqref{moment}, the following holds:
\begin{equation}\label{000}
\tilde{\mu}^{-1}(0)=P \times_G \mu^{-1}(0).
\end{equation}

We use locally conformally symplectic reduction as presented in \cite{hr} or \cite{gopp}  and assume that 0 is a regular value and $G$ acts freely and properly on $\mu^{-1}(0)$. By \eqref{000}, we see that the action of $G$ on $\tilde{\mu}^{-1}(0)$ is still free and proper. Thus, both $\tilde{\mu}^{-1}(0)/G$ and $\mu^{-1}(0)/G$ are manifolds. We are now ready to prove the following compatibility theorem.
\begin{theorem}\label{comp}
Let $G$ be an abelian Lie group acting twisted Hamiltonian on the LCS manifold $(F, \omega, \theta)$ and $P$ a principal $G$-bundle with a fat connection at $\Im(\mu)$.Then $G$ acts twisted Hamiltonian with respect to the coupling form $\Omega$ and assuming all conditions for reduction at 0 are met, the relation between the two reduced manifolds is:  
$$\tilde{\mu}^{-1}(0)/G \simeq P/G \times \mu^{-1}(0)/G.$$
\end{theorem}

\begin{proof}
We define $f: \tilde{\mu}^{-1}(0)/G \rightarrow P/G \times \mu^{-1}(0)/G$, $f(\widehat{[p, f]})=(\hat{p}, \hat{f})$. By using only the way $G$ acts on $F$ and $P \times_G F$, $f$ is well defined with the inverse $l:P/G \times \mu^{-1}(0)/G \rightarrow \tilde{\mu}^{-1}(0)/G$, $l(\hat{p}, \hat{f})=\widehat{[p, f]}$. The map $f$ is the diffeomorphism we want. 
\end{proof} 

\begin{remark}
 It may happen that the base $P/G$  carry a symplectic structure. As previously stated, when $G$ is $S^1$ there is a correspondence between fat connections on a circle bundle and integral symplectic forms of the base. Since both reduced spaces have LCS structures by \cite[Theorem 1]{hr}, it means that a product between a  symplectic manifold and a LCS one may be LCS. This is unlike the LCK case, where under compactness conditions, the product of LCK with K\" ahler is never LCK, as proved in \cite{opv}. 
\end{remark}

\section{Locally conformally K\" ahler bundles}

Here we address the same problems we did so far for locally conformally K\" ahler
manifolds, namely defining LCK fiber bundles and finding conditions for the total space to admit an LCK structure which restricts to the one of the fibers. We necessarily have to consider only the case when the fiber is a complex submanifold of the bundle. 

Recall that a LCK manifold $(M,\omega,J)$ is a Hermitian manifold whose fundamental two form $\omega$ satisfies the LCS definition for a closed one-form $\theta$.

\begin{definition} Let $F \rightarrow M \rightarrow B$ be a fiber bundle with  $(F, \omega, J)$ an LCK manifold, which is a complex submanifold of $(M, J)$. We call it \emph{LCK fiber bundle} if its structure group consists of holomorphic $\omega$-preserving diffeomorphisms. 
\end{definition}

\begin{remark} By adding the extra condition of holomorphicity of the transitions maps, we ensure there exists a canonical complex structure on each fiber $F_b$, which we denote by $J_b$.
\end{remark}

\begin{remark}The problem of finding an LCK structure on $M$ which restricts to $\omega$ on the fiber $F$ does not reduce to regarding $F$ and $M$ only as LCS manifolds. Indeed, if we forget the complex structure and look only for conditions of existence of an LCS extension of $\omega$, which is actually the first step to consider, there is no guarantee the extension we obtain satisfies also $J$ - invariance and positiveness in order to come from an LCK metric. Hence working in the complex brings new constraints for the LCS extensions.
\end{remark}

\smallskip

Let us consider the same setting as in \ref{main}: suppose we have  an LCK fiber bundle with total space $P \times_G F$. We further assume that $P/G$ is a complex manifold, endowed with a complex structure $J_1$.  With these assumptions, we are able to construct the following almost complex structure $\tilde{J}$ on $P \times_G F$:

$$\tilde{J}(X) = \begin{cases} J_b(X) &\mbox{if } X \text{ is a  vertical vector tangent to } F_b,\\
(J_1(\pi_*X))^* &\mbox{if } X \text{ is horizontal}. \end{cases}$$

Clearly, $\tilde{J}$ satisfies the condition $\tilde{J} \circ \tilde{J}=-\id$. We are interested in finding a favorable context for the coupling form to be LCK with respect to $\tilde{J}$. For that, we need to investigate first when $\tilde{J}$ is integrable and hence we look at the Nijenhuis tensor.

We recall that the Nijenhuis tensor defined as 
$$N_J(X, Y)=[X, Y]-[JX, JY]+J[JX, Y]+J[X, JY]$$
vanishes if and only if $J$ is integrable, hence a complex structure. Computing now $N_{\tilde{J}}$, we obtain:
$$N_{\tilde{J}}(X, Y)=0,$$
whenever $X$ and $Y$ are vertical vector fields, since for any point $p$ in $P \times_G F$ with $\pi(p)=b$:
$$N_{\tilde{J}}(X, Y)(p)=N_{J_b}(X_{|F_b}, Y_{|F_b})(p)=0.$$
If $X_p$ and $Y_p$ are horizontal vectors, we can compute $N_{\tilde{J}}(X, Y)_p$ by taking field extensions of $X_p$ and $Y_p$ which are
 lifts of fields from the base, denoted by $X^*$ and $Y^*$.  Then by using the formula of the curvature $[X^*, Y^*]=[X, Y]^*+\mathcal{R}_{\mathcal{H}}(X, Y)$, the definition of $\tilde{J}$ and the fact that $J_1$ is integrable, we get that:
$$N_{\tilde{J}}(X^*, Y^*)=J(\mathcal{R}_{\mathcal{H}}(J_1X, Y) + \mathcal{R}_{\mathcal{H}}(X, J_1Y)) +\mathcal{R}_{\mathcal{H}}(X, Y) - \mathcal{R}_{\mathcal{H}}(J_1X, J_1Y).$$
An LCK form is necessarily of type (1,1), hence $\tilde{J}$-invariant. This condition applied to the coupling form would automatically imply the $\tilde{J}$-invariance of the curvature form with respect to the complex structure on the base, $J_1$. But this will further yield the vanishing of the Nijenhuis tensor applied on horizontal vector fields. 
It remains to deal with the vanishing of $N_{\tilde{J}}(X, Y)$, when $X$ is horizontal and $Y$ is vertical and eventually the positiveness of the coupling form. We cannot control these two conditions, and hence in general {\em the almost complex structure $\tilde J$ is not integrable}. 

It is an open question so far to find sufficient conditions for $\tilde J$ to be integrable.

\bigskip

\noindent{\bf{Acknowledgements.}} I express my thanks to Liviu Ornea for his original suggestion, his constant support and encouragement, to Sergiu Moroianu, Alexandru Oancea and Victor Vuletescu, who have provided me with valuable ideas and to Izu Vaisman for useful remarks and bibliographical data. I also thank the referee for carefully reading a first version of this paper and for his/her recommendations.

\end{document}